\documentclass[11pt]{amsart}
\usepackage{fullpage}
\usepackage{amsmath,amssymb,amsthm,graphicx,url}
\usepackage{booktabs}
\usepackage{epsfig}
\usepackage{tikz}
\usepackage{tikz-qtree}
\usepackage{array}

\usepackage[aligntableaux=center]{ytableau}

\theoremstyle{definition}
\newtheorem{theorem}{Theorem}
\newtheorem{lemma}[theorem]{Lemma}

\newtheorem{defn}{Definition}
\newtheorem{prop}[theorem]{Proposition}
\newtheorem{example}[theorem]{Example}

\newtheorem{remark}[theorem]{Remark}
\newtheorem{question}[theorem]{Question}

\tikzset{every tree node/.style={align=center,anchor=north}}

\newcommand{\myg}{\,\underline{\ }\,}
\newcommand{\HC}{\mathrm{HC}}
\newcommand{\lsp}{\mathrm{LS}}
\newcommand{\lsvec}{\mathrm{V}}
\newcommand{\cs}{\mathrm{CS}}
\newcommand{\scs}{\mathrm{SCS}}

\newcommand{\msvec}{\overline{V}}
\newcommand{\sud}{\mathrm{Sud}}
\DeclareMathOperator{\lcm}{lcm}
\newcommand{\bu}{\boldsymbol{u}}
\newcommand{\bv}[2]{\bu^{#1}\bu^{#2}}
\newcommand{\boldv}{\boldsymbol{v}}
\newcommand{\bw}{\boldsymbol{w}}
\newcommand{\ba}{\boldsymbol{a}}
\newcommand{\bb}{\boldsymbol{b}}
\newcommand{\br}{\boldsymbol{r}}
\newcommand{\bc}{\boldsymbol{c}}
\newcommand{\bx}{\boldsymbol{x}}
\newcommand{\by}{\boldsymbol{y}}
\newcommand{\be}{\boldsymbol{e}}
\newcommand{\boldf}{\boldsymbol{f}}

\def\Big#1{\makebox(0,0){\huge#1}}

\newcommand{\csop}{\mathrm{cs}}
\newcommand{\scsop}{\mathrm{scs}}

\makeatletter
\def\imod#1{\allowbreak\mkern10mu({\operator@font mod}\,\,#1)}
\makeatother

\begin{document}

\title{Orthogonal bases for transportation polytopes applied to latin squares, magic squares and Sudoku boards}

\date{\today}

\author{Gregory S. Warrington}
\address{Dept. of Mathematics and Statistics\\
  University of Vermont \\
  Burlington, VT 05401}
\email{gregory.warrington@uvm.edu}

\thanks{This work was partially supported by a grant from the Simons
 Foundation (\#429570).}

\vspace*{.3in}

\begin{abstract}
  We give a simple construction of an orthogonal basis for the space
  of $m\times n$ matrices with row and column sums equal to zero. This
  vector space corresponds to the affine space naturally associated
  with the Birkhoff polytope, contingency tables and Latin squares. We
  also provide orthogonal bases for the spaces underlying
  magic squares and Sudoku boards. Our construction combines the
  outer (i.e., tensor or dyadic) product on vectors with certain
  rooted, vector-labeled, binary trees. Our bases naturally respect
  the decomposition of a vector space into centrosymmetric and
  skew-centrosymmetric pieces; the bases can be easily modified to
  respect the usual matrix symmetry and skew-symmetry as well.
\end{abstract}

\keywords{orthogonal basis, transportation polytope, Birkhoff
  polytope, contingency table, Latin square, Magic square, Sudoku}

\subjclass[2010]{Primary 52B12; Secondary 05B15, 15A03}

\maketitle


\section{Introduction}
\label{sec:intro}

Matrices with specified row and column sums arise in various contexts
in mathematics: in the definition of the Birkhoff polytope; as
statistical contingency tables; and as Latin squares, magic squares
and Sudoku boards. In this paper, we give simple, explicit
linear-algebraic constructions of orthogonal bases for the vector
spaces underlying these families. We note from the outset that there
are obvious bases for these spaces (see Section~\ref{sec:future}) that
can be orthogonalized by, say, the Gram-Schmidt process. However, we
doubt there is a simple closed-form description of the matrices
resulting from such a process. In addition, our approach has the
advantage of yielding basis vectors that respect natural
decompositions of these vector spaces under various symmetries.

For $m,n\geq 2$, let $\lsvec_{m,n}$ be the $(m-1)(n-1)$-dimensional
subspace of matrices $(x_{ij})$ in $\mathbb{R}^{mn}$ subject to the
$m$ requirements that the row sums are zero and $n$ requirements that
the column sums are zero:
\begin{equation}\label{eq:vmn}
  \sum_{k=1}^n x_{ik} = 0 = \sum_{k=1}^m x_{kj},\quad 1\leq i\leq m,\ 1\leq j\leq n. 
\end{equation}
Our main theorem in this paper is an explicit orthogonal basis for
$\lsvec_{m,n}$. In Section~\ref{sec:vecs} we define for each $k\geq 2$
a set $U(k)$ of $k-1$ vectors. Each element of our basis for
$\lsvec_{m,n}$ can be written as an outer product of an element of
$U(m)$ with an element of $U(n)$. (Recall that the \emph{outer
  product} $\bu \bw$ of $\bu = (u_1,\ldots,u_m)$ with $\bw =
(w_1,\ldots,w_n)$ is the $m\times n$ matrix whose $(i,j)$-th entry is
$u_iw_j$.)

\begin{theorem}\label{thm:main}
  The set of matrices $B_{m,n} = \{\bu\bu':\, \bu\in U(m),\ \bu'\in
  U(n)\}$ provides an orthogonal basis for $\lsvec_{m,n}$.
\end{theorem}

\begin{example}
  In Section~\ref{sec:vecs} it is shown that $U(3) = \{\bu^1 =
  (1,-2,1), \bu^2 = (1,0,-1)\}$. By Theorem~\ref{thm:main}, the
  following four matrices thereby form an orthogonal basis for
  $\lsvec_{3,3}$.  {\small
  \begin{equation*}
    \bv{1}{1} =
    \begin{bmatrix}
      1 & -2 & 1\\
      -2 & 4 & -2\\
      1 & -2 & 1
    \end{bmatrix}, \quad 
    \bv{1}{2} =
    \begin{bmatrix}
      1 & 0 & -1\\
      -2 & 0 & 2\\
      1 & 0 & -1
    \end{bmatrix},  \quad 
    \bv{2}{1} =
    \begin{bmatrix}
      1 & -2 & 1\\
      0 & 0 & 0\\
      -1 & 2 & -1
    \end{bmatrix},  \quad 
    \bv{2}{2} = 
    \begin{bmatrix}
      1 & 0 & -1\\
      0 & 0 & 0\\
      -1 & 0 & 1
    \end{bmatrix}.
  \end{equation*}
  }
\end{example}


The structure of the paper is as follows. In
Section~\ref{sec:transport} we unify the objects being studied under
the umbrellas of transportation polytopes and inside-out polytopes. We
also explain the simple shift from the relevant affine space
containing the objects of interest to the linear subspace
$\lsvec_{m,n}$. In Section~\ref{sec:vecs} we define the sets $U(n)$
and prove Theorem~\ref{thm:main}. In Sections~\ref{sec:magic}
and~\ref{sec:sudoku} we modify our construction of $B_{m,n}$ so as to
provide analogous bases for vector spaces of magic squares and Sudoku
boards, respectively. In Section~\ref{sec:symm} we introduce simple
variations of our bases that respect usual matrix symmetry and
skew-symmetry. Finally, in Section~\ref{sec:future} we explore the
connection to a well-known non-orthogonal basis for $\lsvec_{m,n}$ and
mention some possible directions for further study.

In order to avoid clutter, we will periodically utilize the following
notations: denoting negative numbers by placing a bar of the number;
replacing zeros with underscores; and omitting parentheses and commas
from vectors.

\section{Transportation and inside-out polytopes}
\label{sec:transport}

The examples mentioned in the beginning of the Introduction are
unified by the concept of a transportation (or, transport)
polytope. Transportation polytopes have long been studied in the
fields of statistics, mathematical programming and geometry (see the
following references for different overviews:
~\cite{Diaconis-Gangolli,TPbook,DeLoera}). Let $m,n\geq 1$ and $\br =
(r_1,\ldots,r_m)$, $\bc = (c_1,\ldots,c_n)$ be two vectors of
nonnegative real entries such that $\sum_{k=1}^m r_k = \sum_{k=1}^n
c_k$. The vectors $\br$ and $\bc$ are called \emph{marginals} (or
\emph{margins} or \emph{$1$-marginals}). As in Pak~\cite{Pak}, we
define the \emph{transportation polytope $T(\br,\bc)$} to be the set
of $m\times n$ matrices $(x_{ij})$ over $\mathbb{R}$ satisfying:
\begin{itemize}
\item $x_{ij} \geq 0$ for all $1\leq i,j\leq n$,
\item $\sum_{k=1}^n x_{ik} = r_i$, $1\leq i\leq m$, and
\item $\sum_{k=1}^m x_{kj} = c_j$, $1\leq j\leq n$.
\end{itemize}
The special case of $m=n$ and $\br=\bc=(1,1,\ldots,1)$ is known as the
\emph{Birkhoff polytope} or the \emph{polytope of doubly stochastic
  matrices}. The Birkhoff polytope is one of the most fundamental
of polytopes --- see, e.g.~\cite{Ziegler}.

In many cases, one is interested solely in the lattice points
$\mathbb{Z}^{mn}$ lying in a given polytope. For example, in
statistics a \emph{contingency table} is such a lattice point in the
transportation polytope $T(\br,\bc)$ in which the marginals consist of
integers. Such tables are used to describe the distribution of a
population over two variables. Knowledge of all such integer-lattice
points satisfying a given set of marginals is useful in statistical
tests for significance~\cite{Diaconis-Gangolli}.

There are a number of examples where one is interested in only a
subset of the lattice points lying in a polytope. The cases discussed
in this paper are the following (take $m=n$).
\begin{itemize}
\item Let $\br=\bc=(\binom{n+1}{2},\ldots,\binom{n+1}{2})$. An
  \emph{order-$n$ Latin square} is an element of $\lsp_n = T(\br,\bc)$
  for which each row and column is a permutation of
  $\{1,2,\ldots,n\}$.

\item Let $S\in \mathbb{R}_{\geq 0}$. For $\br=\bc=(S,S,\ldots,S)$ an
  \emph{order-$n$ semi-magic square with magic-sum $S$} is an element
  of $T(\br,\bc)$ for which all $n^2$ entries are distinct. If both
  main-diagonal sums also equal $S$, then such a matrix is not merely
  \emph{semi-magic} but also \emph{magic}. If the entries are
  $\{1,2,\ldots,n^2\}$ (and hence $S = n^2(n^2+1)/2$) then the square
  is said to be \emph{normal} (terminology varies).

\item An $n^2\times n^2$ matrix naturally decomposes into $n^2$
  submatrices, each of size $n\times n$, that simultaneously tile the
  entire grid. An \emph{order-$n^2$ Sudoku board} is an order-$n^2$
  Latin square with one additional property: The entries in
  each of these $n^2$ submatrices must also be a permutation of
  $\{1,2,\ldots,n^2\}$.
\end{itemize}

There is an expansive literature on these families and their many
variations. We direct the reader to~\cite[\S III]{DinitzHandbook} for
an overview of Latin squares and~\cite{BeforeSudoku} for an overview
of Magic squares and Sudoku.

In each of these three cases, the subset of lattice points of interest
can be characterized in a particularly simple way. Recall that a
\emph{hyperplane arrangement} is a finite collection of hyperplanes in
a vector space. A polytope in conjunction with a hyperplane
arrangement is known as an \emph{inside-out polytope}. We can describe
the set of order-$n$ Latin square using an inside-out polytopes as
follows. Denote the $n^2$-dimensional hypercube with sides $1\leq
x_{ij} \leq n$ by $\HC_n$. Let $A_n$ denote the hyperplane arrangement
\begin{equation*}
  \cup_{\substack{i,j,k=1\\j\neq k}}^n \{x_{ij} = x_{ik}\} \bigcup \cup_{\substack{i,j,k=1\\i\neq j}}^n \{x_{ik} = x_{jk}\}.
\end{equation*}
Then the set of order-$n$ Latin squares is the set of those lattice
points lying in the polytope $\lsp_n\cap \HC_n$ and avoiding the
hyperplane arrangement $A_n$. The incorporation of $\HC_n$ in the
definition ensures that all coordinates are between $1$
and $n$. Avoidance of $A_n$ thereby ensures by the pigeonhole principle
that each row and column is a permutation of $\{1,2,\ldots,n\}$.

By suitably modifying the hyperplanes and polytopes considered, Sudoku
boards and magic squares can also be realized as lattice points in
inside-out polytopes. We omit the details and refer the reader
to~\cite{InsideOut} for the general theory of inside-out polytopes.


We now discuss the relationship between the affine spaces in which our
polytopes typically lie and the associated linear subspaces of
$\mathbb{R}^{n^2}$. For concreteness, we focus on Latin squares. We
begin with a well-known fact stated in the introduction. (Proofs can
be found in, for example,~\cite[Lemma 2.3]{DeLoera}
and~\cite{MagicDim}); we include a proof here only for completeness.)

\begin{prop}\label{prop:dim}
  For $m,n\geq 1$, the dimension of $\lsvec_{m,n}$ is $(m-1)(n-1)$.
\end{prop}
\begin{proof}
  $\lsvec_{m,n}$ is a subspace of $\mathbb{R}^{mn}$ that is defined by
  $m+n$ linear equations: $m$ of the form $\sum_{k} x_{ik} = 0$ and
  $n$ of the form $\sum_k x_{kj} = 0$. The sum of the first $m$
  expressions equals the sum of the last $n$ expressions, so we know there
  is a dependency. However, if we omit the $(m+1)$-st equation and order
  the $x_{ij}$ lexicographically, we see that the pivot columns of the
  remaining $m+n-1$ equations are distinct. It follows that
  $\dim(\lsvec_{m,n}) = mn-(m+n-1)$ as desired.
\end{proof}

We have realized Latin squares as points in an $n^2$-dimensional
vector space. However, Latin squares actually live in a proper
subspace. By definition, the polytope $\lsp_n$ consists of points
lying on the $2n$ hyperplanes determined by the marginals. By the
argument of Proposition~\ref{prop:dim}, there are exactly $2n-1$
\emph{independent} conditions among these requirements. It follows
that the polytope $\lsp_n$, and hence order-$n$ Latin squares, live
inside a $n^2-(2n-1) = (n-1)^2$-dimensional affine subspace of
$\mathbb{R}^{n^2}$.

The space $\lsvec_{n,n}$ does not contain any Latin squares as the row
and column sums of elements in $\lsvec_{n,n}$ are forced to be
zero. However, suitable analogues of Latin squares that live in this
vector space are given as follows. Note that $1 + 2 + \cdots + n =
\binom{n+1}{2}$. So if we subtract $\frac{1}{n}\binom{n+1}{2} =
\frac{n+1}{2}$ from each element of a normal Latin square, we get a
matrix in which each row and column is a permutation
of $\{i-\frac{n+1}{2}:\, 1\leq i\leq n\}$. These will be termed
\emph{zeroed Latin squares}.

\begin{remark}
  A similar procedure can be used to translate the
  affine subspace containing an arbitrary transportation polytope to a
  linear subspace of $\lsvec_{m,n}$. If the marginals are initially $\br$
  and $\bc$, then the corresponding translated polytope imposes the
  requirement that $x_{ij} \geq -r_ic_j/(\sum_{i=1}^m r_i)$ rather
  than $x_{ij} \geq 0$.
\end{remark}

It will be useful in our discussion of magic squares in
Section~\ref{sec:magic} to augment $\lsvec_{n,n}$. Let $J_n$ be the
$n\times n$ matrix of all $1$'s. Then the affine space containing
$\lsp_n$ (and hence order-$n$ Latin squares) is an affine subspace of
the $((n-1)^2+1)$-dimensional vector space 
$\langle J_n\rangle \oplus \lsvec_{n,n}$. Note that in this space,
Latin squares will have the extra coordinate equal to $\frac{n+1}{2}$.

\section{Orthogonal bases}
\label{sec:vecs}

For each $n\geq 2$ we will define a set of $n-1$ mutually orthogonal
vectors $U(n) = \{\bu^1,\bu^2,\ldots,\bu^{n-1}\}$. Our first
ingredient will be a vector-valued function, $\bw$, on the integers
greater than $1$.  The second ingredient will be a (rooted), labeled
binary tree, $T_n$, for each $n$. The desired vectors $\bu^i$ will be
the labels of the vertices of our tree $T_n$.

For each $n\geq 2$ we now define $\bw(n) =
(w(n)_0,w(n)_1,\ldots,w(n)_{n-1})$ as follows. If $n$ is odd, then set
$w(n)_i = (n-1)/2$ for $i$ odd and $-(n+1)/2$ for $i$ even. Set $\bw(2)
= (1,-1)$ and $\bw(4) = (1,-1,-1,1)$. For $n$ even and greater than $4$,
say $n=2m$, set $w(n)_i = w(m)_{i\pmod m}$ for $0\leq i\leq
n-1$. Table~\ref{tab:w} shows the values of $\bw$ for small $n$.

\begin{table}[ht]
  \centering
  \caption{Values of $\bw$ function (negatives are denoted by bars).}
  \begin{tabular}{@{}l>{$}c<{$}l>{$}c<{$}@{}} \toprule
    $n$ & \bw(n) & $n$ & \bw(n)\\\midrule
    3 & 1\,\overline{2}\,1 &     4 & 1\,\overline{1}\,\overline{1}\,1\\
    5 & 2\,\overline{3}\,2\,\overline{3}\,2 &     6 & 1\,\overline{2}\,1\,1\,\overline{2}\,1\\
    7 & 3\,\overline{4}\,3\,\overline{4}\,3\,\overline{4}\,3 &     8 & 1\,\overline{1}\,\overline{1}\,1\,1\,\overline{1}\,\overline{1}\,1\\
    9 & 4\,\overline{5}\,4\,\overline{5}\,4\,\overline{5}\,4\,\overline{5}\,4 &    10 & 2\,\overline{3}\,2\,\overline{3}\,2\,2\,\overline{3}\,2\,\overline{3}\,2\\
    11 & 5\,\overline{6}\,5\,\overline{6}\,5\,\overline{6}\,5\,\overline{6}\,5\,\overline{6}\,5 &    12 & 1\,\overline{2}\,1\,1\,\overline{2}\,1\,1\,\overline{2}\,1\,1\,\overline{2}\,1\\\bottomrule
  \end{tabular}
  \label{tab:w}
\end{table}

Given $n\geq 2$, we construct $T_n$ iteratively by depth (i.e.,
distance from the root), starting with a root labeled $\bw(n)$. Once
all vertices of depth at most $d$ have been identified and labeled, we
construct vertices at depth $d+1$ as follows: Suppose a vertex at
depth $d$ has label $\bu=(u_0,u_1,\ldots,u_{n-1})$. Let $I_{+}(\bu) =
\{i_0 < i_1 < \cdots < i_{a'-1}\}$ be the indices for which $u_i > 0$
and $I_{-}(\bu) = \{j_0 < j_1 < \cdots < j_{a''-1}\}$ be those indices
for which $u_j < 0$. If $|I_{+}(\bu)| = a'$ is at least $2$, then we
attach a left child with label $\bu'=(u'_0,u'_1,\ldots,u'_{n-1})$
where
\begin{equation}\label{eq:lsub}
  u'_i = 
  \begin{cases}
    0, & \text{ for } i\not\in I_{+}(\bu),\\
    w(a')_r, & \text{ for } i = i_r.
  \end{cases}
\end{equation}
Similarly, if $|I_{-}(\bu)| = a''$ is at least $2$, we attach a right
child $\bu'' = (u''_0,u''_1,\ldots,u''_{n-1})$ where
\begin{equation}\label{eq:rsub}
  u''_i = 
  \begin{cases}
    0, & \text{ for } i\not\in I_{-}(\bu),\\
    w(a'')_r, & \text{ for } i = j_r.
  \end{cases}
\end{equation}
If $I_{+}(\bu) = I_{-}(\bu) = 1$, then $u$ is a leaf. 

\begin{figure}[h]
\begin{tikzpicture}
  \Tree [.\fbox{$5\,\overline{6}\,5\,\overline{6}\,5\,\overline{6}\,5\,\overline{6}\,5\,\overline{6}\,5$}
       [.\fbox{$1\myg\overline{2}\myg 1\myg 1\myg\overline{2}\myg 1$}
         [.\fbox{$1\myg\myg\myg\overline{1}\myg\overline{1}\myg\myg\myg 1$}
           [.\fbox{$1\myg\myg\myg\myg\myg\myg\myg\myg\myg\overline{1}$} ]
           [.\fbox{$\myg\myg\myg\myg 1\myg\overline{1}\myg\myg\myg\myg$} ]
         ]
         [.\fbox{$\myg\myg 1\myg\myg\myg\myg\myg \overline{1}\myg\myg$} ]
       ]
       [.\fbox{$\myg 2\myg\overline{3}\myg 2\myg\overline{3}\myg 2\myg$}
         [.\fbox{$\myg 1\myg \myg \myg \overline{2}\myg \myg \myg 1\myg $}
           [.\fbox{$\myg 1\myg \myg \myg \myg \myg \myg \myg \overline{1}\myg $} ]
         ] 
         [.\fbox{$\myg \myg \myg 1\myg \myg \myg \overline{1}\myg \myg \myg$} ] ] ]                             
\end{tikzpicture}
\caption{The tree $T_{11}$ (zeros are represented by underscores).}\label{fig:t11}
\end{figure}
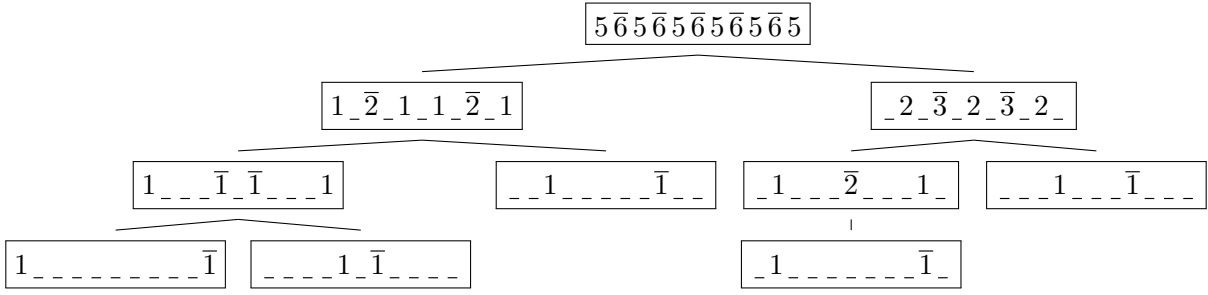

Note that the vector formed by the nonzero entries of any vertex $\bu$
of $T_n$ equals $\bw(k)$ for some $k$. 
If $\bu$ corresponds to $\bw(2)$ after ignoring zeros, then we say
that $\bu^i$ as \emph{skew-symmetric} since $u^i_j = -u^i_j$ for all
$j$. In all other cases, $u^i_j = u^i_{n-1-j}$ for all $j$ and we
refer to such $\bu^i$ as \emph{symmetric}.

\begin{lemma}\label{lem:tree}
  For $n\geq 2$, $T_n$ has $n-1$ vertices, $\lfloor n/2\rfloor$ of
  which are skew-symmetric and $\lfloor (n-1)/2\rfloor$ of which are
  symmetric.
\end{lemma}
\begin{proof}
  First note that
  \begin{equation*}
    \lfloor n/2\rfloor + \lfloor (n-1)/2\rfloor = n-1,
  \end{equation*}
  regardless of the parity of $n$. So the claim regarding the total
  number of vertices of $T_n$ follows directly from the classification
  of their types.

  We enumerate the number of (skew-)symmetric vertices by induction on
  $n$. The base cases of $n\in \{2,3,4\}$ are easily checked directly. We now
  write $n=2^km > 4$ where $m$ is odd. We break into cases according
  to whether $k=0$, $k=1$, $k\geq 2$. As the arithmetic is
  similar in each case, we only work out in detail the case of
  $k=2$. Here the left subtree has the same number of vertices as
  $T_{2^k(m+1)/2}$ and the right subtree the same number of vertices
  as $T_{2^k(m-1)/2}$. It follows by induction that there are
  \begin{equation*}
    \left\lfloor \frac{2^k\frac{m+1}{2}}{2}\right\rfloor + 
    \left\lfloor \frac{2^k\frac{m-1}{2}}{2}\right\rfloor =
    2^{k-2}(m+1) + 2^{k-2}(m-1) = 2^{k-1}m
  \end{equation*}
  skew-symmetric vertices and 
  \begin{equation*}
    \left\lfloor \frac{2^k\frac{m+1}{2}-1}{2}\right\rfloor + 
    \left\lfloor \frac{2^k\frac{m-1}{2}-1}{2}\right\rfloor = 
    2^{k-2}(m+1)-1 + 2^{k-2}(m-1) -1 = 2^{k-1}m - 2
  \end{equation*}
  symmetric vertices contained in these two subtrees. Together with
  the root (which is symmetric for all $n > 2$), we obtain
  $2^{k-1}m=\lfloor \frac{n}{2}\rfloor$ skew-symmetric vertices and
  $(2^{k-1}m-2)+1 = \lfloor \frac{n-1}{2}\rfloor$ symmetric vertices
  in $T_n$. The other cases are similar.
\end{proof}

\begin{defn}
Let $n\geq 2$. Define the set $U(n) =
\{\bu^1,\bu^2,\ldots,\bu^{n-1}\}$ by setting $\bu^i$ to be the $i$-th
vertex encountered while performing a depth-first traversal of $T_n$
(choose the left child before the right child).
\end{defn}

Note that while $\bu^i$ depends on $n$, this dependence is typically
omitted from the notation. When we find it useful to explicitly
indicate $n$, we will write $\bu^{n;i}$ for $\bu^i$. Continuing our
example with $n=11$ from Figure~\ref{fig:t11} we see immediately that
\begin{alignat*}{3}
& \bu^1 = 5\,\overline{6}\,5\,\overline{6}\,5\,\overline{6}\,5\,\overline{6}\,5\,\overline{6}\,5, 
&\quad & \bu^2 = 1\myg\overline{2}\myg 1\myg 1\myg\overline{2}\myg 1,\\
& \bu^3 = 1\myg\myg\myg\overline{1}\myg\overline{1}\myg\myg\myg 1, 
&\quad & \bu^4 = 1\myg\myg\myg\myg\myg\myg\myg\myg\myg\overline{1},\\
& \bu^5 = \myg\myg\myg\myg 1\myg\overline{1}\myg\myg\myg\myg, 
&\quad & \bu^6 = \myg\myg 1\myg\myg\myg\myg\myg \overline{1}\myg\myg\\
& \bu^7 = \myg 2\myg\overline{3}\myg 2\myg\overline{3}\myg 2\myg, 
&\quad & \bu^8 = \myg 1\myg \myg \myg \overline{2}\myg \myg \myg 1\myg,\\
& \bu^9 = \myg 1\myg \myg \myg \myg \myg \myg \myg \overline{1}\myg, 
&\quad & \bu^{10} = \myg \myg \myg 1\myg \myg \myg \overline{1}\myg \myg \myg.
\end{alignat*}

\begin{lemma}\label{lem:ufacts}
  For $n\geq 2$ and $\bu\in U(n)$,
  \begin{enumerate}
    \item $\sum_{i=0}^{n-1} u_i = 0$ \label{lem:uzero} and 
    \item $|\{u_i > 0\}| = |\{u_i < 0\}| = 1$. \label{lem:twodistinct}
  \end{enumerate}
\end{lemma}
\begin{proof}
  It follows immediately from equations~\eqref{eq:lsub}
  and~\eqref{eq:rsub} that given $\bu\in U(n)$, there exists a $k \geq
  2$ such that $\sum_{i=0}^{n-1} u_i = \sum_{i=0}^{k-1} w(k)_i$. That
  this latter sum is zero follows by induction. 

  For the second part, note that each $\bu$ equals $\bw(k)$ for some
  $k$ once zeros are ignored. That each $\bw(k)$ has a unique positive
  value and a unique negative value follows by induction.
\end{proof}

\begin{prop}\label{prop:uorthog}
  The set $U(n)$ is an orthogonal set.
\end{prop}
\begin{proof}
  Consider vectors $\bu^j,\bu^k\in U(n)$. Without loss of generality
  we assume $j < k$. If $\bu^j$ does not lie on the path from $\bu^k$
  to the root, then the orthogonality follows trivially since for each
  index $0 \leq i \leq n-1$, either $u^j_i = 0$ or $u^k_i = 0$ (or
  both).  So suppose $\bu^j$ does lie on the path from $\bu^k$ to the
  root. It follows from the definitions of $I_{+}$, $I_{-}$,
  Lemma~\ref{lem:ufacts}.\ref{lem:twodistinct}, and
  equations~\eqref{eq:lsub} and~\eqref{eq:rsub} that for all pairs
  $u^k_i, u^k_{i'} \neq 0$, we have $u^j_i = u^j_{i'}$. But then
  $\bu^j\cdot \bu^k$ is a scalar multiple of $\sum_{i=0}^{n-1} u^k_i$,
  which by Lemma~\ref{lem:ufacts}.\ref{lem:uzero} is zero.
\end{proof}

The \emph{outer product} of two vectors $\boldsymbol{a} =
(a_0,a_1,\ldots,a_{m-1})$, $\boldsymbol{b} =
(b_0,b_1,\ldots,b_{n-1})$ is given by
\begin{equation*}\label{eq:dyadic}
  \ba\bb = \ba\otimes \bb = \begin{bmatrix} a_0b_0 & a_0b_1 & \hdots & a_0b_{n-1}\\
    a_1b_0 & a_1b_1 & \hdots & a_1b_{n-1}\\
    \vdots & \vdots & \ddots & \vdots\\
    a_{m-1}b_0 & a_{m-1}b_1 & \hdots & a_{m-1}b_{n-1}
  \end{bmatrix}.
\end{equation*}
This product, which we denote by juxtaposition, can be viewed
(treating $\ba$ and $\bb$ as row vectors) as a matrix product
$\ba^T\bb$ or as a \emph{dyadic product}. We are now ready to prove
Theorem~\ref{thm:main}.

\begin{proof}[Proof of Theorem~\ref{thm:main}]
  Note that each row and column of $\bv{i}{j}$ is a scalar multiple of
  a vector whose entries sum to zero. It follows immediately that each
  element of $B_{m,n}$ lies in $\lsvec_{m,n}$.

  We first show that the cardinality of the set $B_{m,n}$ is $(m-1)(n-1)$, the
  dimension of $\lsvec_{m,n}$. So suppose that $\bv{i}{j} = \bv{k}{\ell}$ for
  some $1\leq i, k\leq m-1$ and $1\leq j, \ell\leq n-1$. Then,
  in particular, the $a$-th row of $\bv{i}{j}$ equals the $a$-th row
  of $\bv{k}{\ell}$ for each $1\leq a\leq m$. This implies in turn
  that
  \begin{equation}
    u^i_a\bu^j = (u^i_au^j_0,u^i_au^j_1,\ldots, u^i_au^j_{n-1}) \text{ equals }
    u^k_a\bu^\ell = (u^k_au^\ell_0,u^k_au^\ell_1,\ldots, u^k_au^\ell_{n-1}).
  \end{equation}
  Since the $U(n)$ is an orthogonal set by
  Proposition~\ref{prop:uorthog}, it follows that either
  $u^i_a=u^k_a=0$ or that $\bu^j = \bu^\ell$. Since $\bu^i$ is not the
  zero vector, it follows that we can pick an $a$ for which
  $u^i_a\neq 0$. We conclude that $\bu^j = \bu^\ell$ and hence that
  $j=\ell$ by orthogonality. In turn, this tells us that $u^i_a=u^k_a$
  for all $a$. Again by Proposition~\ref{prop:uorthog}, we conclude $i
  = k$. We conclude that the $(m-1)(n-1)$ products $\bv{i}{j}$ are all
  distinct.

  If follows immediately from the definition of the $\bv{i}{j}$ and
  Proposition~\ref{prop:uorthog} that none of the $\bv{i}{j}$ are the
  zero vector. Hence, to show linear independence, it suffices to show
  that they are pairwise orthogonal. Since we know the dimension of
  $\lsvec_{m,n}$ to be $(m-1)(n-1)$, we can then conclude that the elements
  of $B_{m,n}$ form a basis, as desired. So: Let $1\leq i, k\leq m-1$ and
  $1\leq j, \ell\leq n-1$. Then
  \begin{equation*}
    \bv{i}{j}\cdot \bv{k}{\ell} = \sum_{a=1}^m\sum_{b=1}^n (\bv{i}{j})_{a,b}(\bv{k}{\ell})_{a,b}
    = \sum_{a=1}^m u^i_au^k_a\sum_{b=1}^n u^j_bu^\ell_b = (\bu^i\cdot \bu^k)(\bu^j\cdot \bu^\ell).
  \end{equation*}
  By the orthogonality of the $\bu^i$, we conclude immediately that
  $\bv{i}{j}$ and $\bv{k}{\ell}$ are orthogonal whenever $(i,j) \neq
  (k,\ell)$. This completes the proof.
\end{proof}

\begin{example}
  The zeroed Latin square $\begin{ytableau}1 & -1 & 0\\-1 & 0 & 1\\0 & 1 &
    -1\end{ytableau}$ can be written as a linear combination of two of
    the elements of $B_{3,3}$:
    \begin{equation*}
      \begin{bmatrix}1 & -1 & 0\\-1 & 0 & 1\\0 & 1 &
        -1\end{bmatrix} = \frac{1}{2}\left(\bv{1}{2} + \bv{2}{1}\right) = \frac{1}{2}
        \begin{bmatrix} 1 & 0 & -1\\-2 & 0 & 2\\1 & 0 & -1\end{bmatrix} +
          \frac{1}{2}\begin{bmatrix} 1 & -2 & 1\\ 0 & 0 & 0\\-1 & 2 & -1\end{bmatrix}.
    \end{equation*}
    Table~\ref{tab:ex} lists the expansions for all twelve order-$3$
    zeroed Latin squares.
\end{example}

\begin{table}[ht]
  \centering
  \caption{Expansions of order-$3$ zeroed Latin squares in the basis $B_{3,3}$.}
  \begin{tabular}{@{}c >{$}c<{$} >{$}c<{$} >{$}c<{$} >{$}c<{$} c c >{$}c<{$} >{$}c<{$} >{$}c<{$} >{$}c<{$}@{}}\toprule
     & \bv{1}{1} & \bv{1}{2} & \bv{2}{1} & \bv{2}{2} & &  & \bv{1}{1} & \bv{1}{2} & \bv{2}{1} & \bv{2}{2}\\
    \cmidrule{1-5} \cmidrule{7-11}
\begin{ytableau} 1 & -1 &  0\\ -1 &  0 &  1\\  0 &  1 & -1\end{ytableau} &  0 & \frac{1}{2} & \frac{1}{2} &  0 &

& \begin{ytableau} 0 &  1 & -1\\ -1 &  0 &  1\\  1 & -1 &  0\end{ytableau} &  0 & \frac{1}{2} &  -\frac{1}{2} &  0\\[0.8cm]

\begin{ytableau} 0 & -1 &  1\\  1 &  0 & -1\\ -1 &  1 &  0\end{ytableau} &  0 &  -\frac{1}{2} & \frac{1}{2} &  0 &

& \begin{ytableau}-1 &  1 &  0\\  1 &  0 & -1\\  0 & -1 &  1\end{ytableau} &  0 &  -\frac{1}{2} &  -\frac{1}{2} &  0\\[0.8cm]

\begin{ytableau} 0 & -1 &  1\\ -1 &  1 &  0\\  1 &  0 & -1\end{ytableau} & \frac{1}{4}& \frac{1}{4}& \frac{1}{4}&  -\frac{3}{4} &
& \begin{ytableau} 1 &  0 & -1\\ -1 &  1 &  0\\  0 & -1 &  1\end{ytableau} & \frac{1}{4}& \frac{1}{4}&  -\frac{1}{4} & \frac{3}{4}\\[0.8cm]
\begin{ytableau} 1 & -1 &  0\\  0 &  1 & -1\\ -1 &  0 &  1\end{ytableau} & \frac{1}{4}&  -\frac{1}{4} & \frac{1}{4}& \frac{3}{4} &
& \begin{ytableau}-1 &  0 &  1\\  0 &  1 & -1\\  1 & -1 &  0\end{ytableau} & \frac{1}{4}&  -\frac{1}{4} &  -\frac{1}{4} &  -\frac{3}{4}\\[0.8cm]

\begin{ytableau} 1 &  0 & -1\\  0 & -1 &  1\\ -1 &  1 &  0\end{ytableau} &  -\frac{1}{4} & \frac{1}{4}& \frac{1}{4}& \frac{3}{4} &
& \begin{ytableau}-1 &  1 &  0\\  0 & -1 &  1\\  1 &  0 & -1\end{ytableau} &  -\frac{1}{4} & \frac{1}{4}&  -\frac{1}{4} &  -\frac{3}{4}\\[0.8cm]
\begin{ytableau}-1 &  0 &  1\\  1 & -1 &  0\\  0 &  1 & -1\end{ytableau} &  -\frac{1}{4} &  -\frac{1}{4} & \frac{1}{4}&  -\frac{3}{4} &
& \begin{ytableau} 0 &  1 & -1\\  1 & -1 &  0\\ -1 &  0 &  1\end{ytableau} &  -\frac{1}{4} &  -\frac{1}{4} &  -\frac{1}{4} & \frac{3}{4}\\\bottomrule
  \end{tabular}
  \label{tab:ex}
\end{table}

The transportation polytopes defined at the beginning of this section
are sometimes referred to as \emph{$2$-way transportation
  polytopes}. There are multiple ways to generalize to dimensions $d >
2$ by placing various constraints on $p_1\times p_2\times \cdots
\times p_d$ arrays of real numbers. In the case where all of the
$1$-marginals (i.e., the sums over all but one index) are specified,
we obtain an affine, $d$-dimensional analogue $V_{p_1,p_2,\ldots,p_d}$
of $\lsvec_{m,n}$ whose dimension is $\prod_{i=1}^d
(p_i-1)$. Arguments analogous to those given in this section show that
the $d$-fold products $\bu^{i_1}\bu^{i_2}\cdots\bu^{i_d}$ give an
orthogonal basis for $V_{p_1,p_2,\ldots,p_d}$.

\section{Magic squares}
\label{sec:magic}

In analogy with our terminology for Latin squares, we will use the
term \emph{zeroed magic square} to refer to a magic square for which
all row, column and main-diagonal sums are $0$.

Zeroed magic squares lie in a codimension-$2$ subspace of $\lsvec_{n,n}$
obtained by imposing the two additional constraints that $\sum_{i=1}^n
x_{i,i} = 0 = \sum_{i=1}^n x_{i,n-i+1}$. (We leave it to the
reader to check that these conditions are independent of each other
and of the Latin square conditions.) Let $\msvec_n$ denote this
codimension-$2$ subspace.

\begin{lemma}
  If $\bu^i,\bu^j\in U(n)$, $i\neq j$, then $\bv{i}{j}\in \msvec_n$.
\end{lemma}
\begin{proof}
  This follows by an argument analogous to that found in the proof of
  Proposition~\ref{prop:uorthog}. If $\bu^i$ and $\bu^j$ are not
  related in $T_n$, then all diagonal and anti-diagonal entries of
  $\bv{i}{j}$ are zero. Otherwise, both the diagonal
  $(x_{1,1},x_{2,2},\ldots,x_{n,n})$ and anti-diagonal
  $(x_{n,1},x_{n-1,2},\ldots,x_{1,n})$ are scalar multiples of either
  $\bu^i$ or $\bu^j$ (depending on which is closer to the root). Since
  the sum of the entries of each $\bu^k$ is $0$ by
  Lemma~\ref{lem:ufacts}.\ref{lem:uzero}, the result follows.
\end{proof}

In light of the above lemma, we will construct a basis for $\msvec_n$
by taking $\{\bv{i}{j}:\, i < j\}$ and adjoining $(n-1)-2$ vectors
generated from the $n-1$ vectors of the form $\bv{i}{i}$. Write $k =
\lfloor n/2\rfloor$, $k' = \lfloor (n-1)/2\rfloor$, and (recalling the
definition from Section~\ref{sec:vecs} and Lemma~\ref{lem:tree}),
write $\{\bx^1,\bx^2,\ldots,\bx^k\}$ for the set $\{\bv{i}{i}:\,
\bu^i\text{ is skew-symmetric}\}$; write
$\{\by^1,\by^2,\ldots,\by^{k'}\}$ for the set $\{\bv{i}{i}:\,
\bu^i\text{ is symmetric}\}$.

For $1\leq i\leq k-1$, define
\begin{equation*}
  \overline{\bx}^i = \sum_{j=1}^{k} u^{k;i}_j \bx^j.
\end{equation*}
Since the nonzero diagonal entries of the $\bx^j$ are all $1$'s and the
nonzero anti-diagonal entries of the $\bx^j$ are all $-1$'s, it follows
immediately from Lemma~\ref{lem:ufacts}.\ref{lem:uzero} that each $\overline{\bx}^i$
lies in $\msvec_n$. Linear independence of these $k-1$ vectors will
follow from the orthogonality arguments contained in the proof of
Theorem~\ref{thm:msb}.

We can proceed similarly in finding a codimension-$1$ subspace of the
span of the $\by^i$, except we need to account for the fact that the
diagonal sums of the $\by^i$ vary. Let
\begin{equation*}
  \ell_i = \sum_{j=1}^n \by^i_{jj}, 1\leq i\leq k'\text{ and }\ell = \lcm\{\ell_1,\ldots,\ell_{k'}\}.
\end{equation*}
Then for $1\leq i\leq k'-1$ we set
\begin{equation*}
  \overline{\by}^i = \sum_{j=1}^{k'} \frac{\ell}{\ell_j} \bu_j^{k';i} \by^j.
\end{equation*}

\begin{theorem}\label{thm:msb}
  The set
  \begin{equation}\label{eq:msb}
    \{\bv{i}{j}:\,i<j\} \cup \{\overline{\bx}^i\}_{i=1}^{k-1} \cup \{\overline{\by}^i\}_{i=1}^{k'-1}
  \end{equation}
  is an orthogonal basis for $\msvec_n$.
\end{theorem}
\begin{proof}
  We already know that $B_{n,n}$ is an orthogonal set. Note that each
  $\overline{\bx}^i$ is a linear combination of elements from
  $\{\bx^1,\ldots,\bx^k\}$ and each $\overline{\by}^j$ is a linear
  combination of elements from the disjoint set
  $\{\by^1,\ldots,\by^{k'}\}$. To prove orthogonality of the entire set,
  it therefore to suffices to show that the $\overline{\bx}^i$ and
  mutually orthogonal and that the $\overline{\by}^j$ are mutually
  orthogonal. Since we will have identified in equation~\eqref{eq:msb}
  $(n-1)^2-2$ linearly independent vectors in a
  $((n-1)^2-2)$-dimensional vector space, the claim will follow.

  We have
  \begin{align*}
    \overline{\bx}^i\cdot \overline{\bx}^j &= \left(\sum_{a=1}^k u_a^{k;i}\bx^a\right)\cdot \left(\sum_{b=1}^k u_b^{k;j}\bx^b\right)\\
    &= \sum_{a=1}^k \|\bx^a\|^2 u_a^{k;i}u_a^{k;j}\\
    &= 4(\bu^{k;i}\cdot\bu^{k;j}) = 4\delta_{i,j},
  \end{align*}
  where $\delta_{i,j}$ is the Kronecker delta. So
  $\{\overline{\bx}^1,\ldots,\overline{\bx}^k\}$ is an orthogonal set and, since
  each $\overline{\bx}^i$ is easily seen to be nonzero, it follows that it
  is a linearly independent set. To prove the analogous result for the
  $\overline{\by}^i$ we first note that for each $\by^i$, there exists a
  $\bu\in U(n)$ such that $\by^i = \bu\bu$. We also note that 
  \begin{equation*}
    \ell_i = \sum_{j=1}^n \by^i_{jj} = \sum_{j=1}^n (\bu_j)^2 = \|\bu\|^2 = \sqrt{\|\bu\|^2\|\bu\|^2}
    = \sqrt{\sum_{a=1}^n u_a^2 \sum_{b=1}^n u_b^2} = \sqrt{\sum_{a,b=1}^n (u_au_b)^2} = \|\by^i\|.
  \end{equation*}
  Hence,
  \begin{align*}
    \overline{\by}^i\cdot \overline{\by}^j &= \left(\sum_{a=1}^{k'} \frac{\ell}{\ell_a} u_a^{k';i} \by^a\right)\cdot
    \left(\sum_{b=1}^{k'} \frac{\ell}{\ell_b} u_b^{k';j} \by^b\right)\\
    &= \sum_{a=1}^{k'} \frac{\ell^2}{\ell_a^2}u_a^{k';i}u_a^{k';j} \|\by^a\|^2
    = \ell^2\sum_{a=1}^{k'} u_a^{k';i}u_a^{k';j} = \ell^2 (\bu^{k';i}\cdot \bu^{k';j}) = \ell^2\delta_{i,j}.
  \end{align*}
\end{proof}

\begin{example}
  For $n=3$,
  \begin{equation*}
    \msvec_3 = \langle \bv{1}{2},\bv{2}{1}\rangle = 
    \left\langle \begin{bmatrix}
      1 & 0 & -1\\
      -2 & 0 & 2\\
      1 & 0 & -1
    \end{bmatrix},
    \begin{bmatrix}
      1 & -2 & 1\\
      0 & 0 & 0\\
      -1 & 2 & -1
    \end{bmatrix}\right\rangle.
  \end{equation*}
\end{example}

\begin{example}
  Consider $n=6$. We have
  \begin{equation*}
    U(6) = \{(1,-2,1,1,-2,1), (1,0,-1,-1,0,1), (1,0,0,0,0,-1), (0,0,1,-1,0,0), (0,1,0,0,-1,0)\}.
  \end{equation*}
  So $\bu^1$ and $\bu^2$ are symmetric while $\bu^3$, $\bu^4$ and $\bu^5$ are skew-symmetric.
  It follows that
  {\small
  \begin{equation*}
    \bx^1 = 
    \begin{bmatrix}
        1 & 0 & 0 & 0 & 0 & -1\\
        0 & 0 & 0 & 0 & 0 & 0\\
        0 & 0 & 0 & 0 & 0 & 0\\
        0 & 0 & 0 & 0 & 0 & 0\\
        0 & 0 & 0 & 0 & 0 & 0\\
        1 & 0 & 0 & 0 & 0 & -1
    \end{bmatrix}, \quad 
    \bx^2 = 
    \begin{bmatrix}
        0 & 0 & 0 & 0 & 0 & 0\\
        0 & 0 & 0 & 0 & 0 & 0\\
        0 & 0 & 1 & -1 & 0 & 0\\
        0 & 0 & -1 & 1 & 0 & 0\\
        0 & 0 & 0 & 0 & 0 & 0\\
        0 & 0 & 0 & 0 & 0 & 0
    \end{bmatrix}, \quad \text{and}\quad 
    \bx^3 = 
    \begin{bmatrix}
        0 & 0 & 0 & 0 & 0 & 0\\
        0 & 1 & 0 & 0 & -1 & 0\\
        0 & 0 & 0 & 0 & 0 & 0\\
        0 & 0 & 0 & 0 & 0 & 0\\
        0 & -1 & 0 & 0 & 1 & 0\\
        0 & 0 & 0 & 0 & 0 & 0
    \end{bmatrix} 
  \end{equation*}
  }
  and 
  \begin{equation*}
  {\small
    \by^1 = 
    \begin{bmatrix}
      1 & -2 & 1 & 1 & -2 & 1\\
      -2 & 4 & -2 & -2 & 4 & -2\\
      1 & -2 & 1 & 1 & -2 & 1\\
      1 & -2 & 1 & 1 & -2 & 1\\
      -2 & 4 & -2 & -2 & 4 & -2\\
      1 & -2 & 1 & 1 & -2 & 1
    \end{bmatrix}
    \quad \text{and}\quad  
    \by^2 = 
    \begin{bmatrix}
      1 & 0 & -1 & -1 & 0 & 1\\
      0 & 0 & 0 & 0 & 0 & 0\\
      -1 & 0 & 1 & 1 & 0 & -1\\
      -1 & 0 & 1 & 1 & 0 & -1\\
      0 & 0 & 0 & 0 & 0 & 0\\
      1 & 0 & -1 & -1 & 0 & 1
    \end{bmatrix}.
  }
  \end{equation*}
  We find from considering $\bu^{3;1} = (1,-2,1)$ and $\bu^{3;2}=(1,0,-1)$ that 
 \begin{equation*}
   {\small
   \overline{\bx}^1 = \bx^1-2\bx^2+\bx^3 = 
   \begin{bmatrix}
     1 &  0 & 0 & 0 & 0 &-1\\
     0 &  1 & 0 & 0 & -1 & 0\\
     0 &  0 & -2 & 2 & 0 & 0\\
     0 &  0 & 2 & -2 & 0 & 0\\
     0 &  -1 & 0 & 0 & 1 & 0\\
     -1 &  0 & 0 & 0 & 0 & 1
   \end{bmatrix}, 
   \overline{\bx}^2 = \bx^1 - \bx^3
   \begin{bmatrix}
     1 &  0 & 0 & 0 & 0 &-1\\
     0 &  -1 & 0 & 0 & 1 & 0\\
     0 &  0 & 0 & 0 & 0 & 0\\
     0 &  0 & 0 & 0 & 0 & 0\\
     0 &  1 & 0 & 0 & -1 & 0\\
     -1 &  0 & 0 & 0 & 0 & 1
   \end{bmatrix}.
   }
 \end{equation*}
 Similarly, after computing $\ell_1 = 12$, $\ell_2 = 4$, $\ell =
 \lcm(12,4)=12$ and $\bu^{2;1} = (1,-1)$, we find that 
\begin{equation*}
  {\small
   \overline{y}^1 = \frac{12}{12}\cdot 1\cdot \by^1 + \frac{12}{4}\cdot (-1)\cdot \by^2 =
   \begin{bmatrix}
     -2 & -2 & 4 & 4 &-2 &-2\\
     -2 &  4 &-2 &-2 & 4 &-2\\
     4 & -2 &-2 &-2 &-2 & 4\\
     4 & -2 &-2 &-2 &-2 & 4\\
     -2 &  4 &-2 &-2 & 4 &-2\\
     -2 & -2 & 4 & 4 &-2 &-2
   \end{bmatrix}.
  }
 \end{equation*}
\end{example}

\section{Sudoku}
\label{sec:sudoku}

Let $\sud_{n^2}$ denote the subspace of $V_{n^2,n^2}$ arising from
requiring that the $n^2$ $n\times n$ submatrices that tile an
$n^2\times n^2$ matrix all sum to zero. Zeroed Sudoku boards are
certain points of $\mathbb{Z}^{n^4}$ lying in $\sud_{n^2}$.

\begin{prop}\label{prop:suddim}
  The dimension of $\sud_{n^2}$ is $n^4 - (2n^2-1) - (n-1)^2 = n(n-1)^2(n+2) = n^2(n-1)^2 + 2n(n-1)^2$.
\end{prop}
\begin{proof}
  Order the $n^4$ variables in an $n^2\times n^2$ matrix by reading
  rows from left to right, starting with the top row and working
  towards the bottom. As in the proof of Proposition~\ref{prop:dim},
  omit the condition corresponding to the first column. This leaves
  $2n^2-1$ independent row/column conditions. Consider an $n\times n$ square
  flush with the top edge. That its entries sum to zero follows from
  the conditions on the $(n-1)$ squares lying below it along with the
  conditions on its $n-1$ columns. That the condition on any $n\times
  n$ square flush with the leftmost column is redundant follows
  similarly. By removing these $2n-1$ conditions, we are left with
  $(n-1)^2$ conditions on the $n\times n$ squares. The remaining
  conditions have mutually distinct pivot columns, so must be linearly
  independent.
\end{proof}

Let $\be^i$ be the (length-$n$) vector of all zeros except for a $1$
in position $i$. Let $\boldf$ be the length $n$ vector of all $1$'s.

\begin{theorem}
  A basis for $\sud_{n^2}$ is given by
  \begin{multline}\label{eq:sudbasis}
    \{\be^i\be^j\otimes \bu^k\bu^\ell:\, 1\leq i,j\leq n,\ 1\leq k,\ell\leq n-1\}\ \cup\\
    \{\bu^i\be^j\otimes \boldf\bu^k:\, 1\leq i,k\leq n-1,\ 1\leq j\leq n\}\ \cup\\
    \{\be^j\bu^i\otimes \bu^k\boldf:\, 1\leq i,k\leq n-1,\ 1\leq j\leq n\}.
  \end{multline}
\end{theorem}
\begin{proof}
  For ease of reference, refer to the three sets in
  equation~\eqref{eq:sudbasis} as $A$, $B$ and $C$, respectively.
  Each vector listed is manifestly non-zero. To show linear
  independence, it therefore suffices to show that the vectors are
  pairwise orthogonal. And since the first set yields $n^2(n-1)^2$
  vectors while the second and third each yield $n(n-1)^2$, once linear
  independence is shown, that the set is spanning will follow
  automatically from our dimension count in
  Proposition~\ref{prop:suddim}.

  Consider two arbitrary, distinct vectors from
  equation~\eqref{eq:sudbasis}. Our proof of orthogonality is broken
  into six parts according to the which of the sets $A$, $B$ or $C$
  these vectors live in. For the reader's convenience we illustrate in
  equation~\eqref{eq:sud} one example matrix from each of the sets $A$,
  $B$ and $C$.

  \begin{enumerate}
    \item \emph{Both in $A$}.  Consider the dot product of
      $e^ie^j\otimes \bu^k\bu^\ell$ and $e^{i'}e^{j'}\otimes
      \bu^{k'}\bu^{\ell'}$. If $i\neq i'$ or $j\neq j'$, then each of the
      $n^4$ coordinates is $0$ for at least one of the vectors. If
      $i=i'$ and $j=j'$, then we are reduced to checking orthogonality
      in $B_{n,n}$, which we have already done in Theorem~\ref{thm:main}.
    \item \emph{Both in $B$}. Consider the dot product of
      $\bu^i\be^j\otimes \boldf\bu^k$ and $\bu^{i'}\be^{j'}\otimes
      \boldf\bu^{k'}$. If $j\neq j'$, then there are no nonzero
      entries in common, so assume $j=j'$. The presence of $\boldf$,
      as far as the dot product is concerned, simply multiplies the
      final result by $n$. So we are reduced to considering the dot
      product of $\bu^i\bu^k$ and $\bu^{i'}\bu^{k'}$. This is known to
      be $\delta_{(i,j),(i',k')}$ by Theorem~\ref{thm:main}.
    \item \emph{Both in $C$}. By taking the transpose of each
      matrix, this reduces to the previous case.

    \item \emph{One in $A$, one in $B$}. Consider the dot product
      of $\be^i\be^j\otimes \bu^k\bu^\ell$ and $\bu^{i'}\be^{j'}\otimes
      \boldf\bu^{k'}$. If $j' \neq j$ or $u^{i'}_i = 0$, then the
      result is immediately zero. Otherwise, we are reduced to
      considering the dot product of $\bu^k\bu^\ell$ and
      $\boldf\bu^{k'}$. Since the set $U(n)$ is orthogonal, the dot
      product will be zero unless $\ell = k'$. In this case, the
      result will be $n\|\bu^\ell\|^2\sum_a u^k_a = 0$.

    \item \emph{One in $A$, one in $C$}. By taking the transpose of
      each matrix, this reduces to the previous case.

    \item \emph{One in $B$, one in $C$}. Consider the dot product of
      $\bu^i\be^j\otimes \boldf\bu^k$ and $\be^{i'}\bu^{j'}\otimes
      \bu^{k'}\boldf$. The result is immediately zero if $u^i_{i'} =
      0$ or $u^{j'}_j = 0$. Otherwise, the problem reduces to the dot
      product of $\boldf\bu^k$ and $\bu^{k'}\boldf$. This is easily
      computed as
      \begin{equation*}
        \sum_{a,b} (\boldf\bu^i)_{a,b}(\bu^j\boldf)_{a,b} =
        \sum_{a,b} \bu^i_b\bu^j_a = \sum_b \bu^i_b \sum_a \bu^j_a = 0
      \end{equation*}
      by Lemma~\ref{lem:uzero}.
  \end{enumerate}
  This completes the proof.
\end{proof}

\begin{example}
  Let $n=3$. Below we illustrate a basis element arising from each of
  the three sets of equation~\eqref{eq:sudbasis}.  {\small
\begin{multline}\label{eq:sud}
  \be^1\be^2\otimes \bu^2\bu^2 = 
    \left[\begin{array}{ccc|ccc|ccc}
    & & & 1 & 0 & -1 & & & \\
    & $\Big{0}$ & & 0 & 0 & 0 & & \Big{0}& \\
    & & & -1 & 0 & 1 & & & \\\hline
    & & & & & & & & \\
    & \Big{0} & & & \Big{0} & & & \Big{0} & \\
    & & & & & & & & \\\hline
    & & & & & & & & \\
    & \Big{0} & & & \Big{0} & & & \Big{0} & \\
    & & & & & & & &
      \end{array}
        \right],
  \bu^2\be^2\otimes \boldf\bu^1 =
    \left[\begin{array}{ccc|ccc|ccc}
    & & & 1 & -2 & 1 & & & \\
    & \Big{0} & & 1 & -2 & 1 & & \Big{0} & \\
    & & & 1 & -2 & 1 & & & \\ \hline
    & & & & & & & & \\
    & \Big{0} & & & \Big{0} & & & \Big{0} & \\
    & & & & & & & & \\ \hline
    & & & -1 & 2 & -1 & & & \\
    & \Big{0} & & -1 & 2 & -1 & & \Big{0} & \\
    & & & -1 & 2 & -1 & & & 
  \end{array}\right], \text{ and }\\
  \be^1\bu^1\otimes \bu^1\boldf =
    \left[\begin{array}{ccc|ccc|ccc}
    1 & 1 & 1 & -2 & -2 & -2 & 1 & 1 & 1\\
    -2 & -2 & -2 & 4 & 4 & 4 & -2 & -2 & -2\\
    1 & 1 & 1 & -2 & -2 & -2 & 1 & 1 & 1\\ \hline
    & & & & & & & & \\
    & \Big{0} & & & \Big{0} & & & \Big{0} & \\
    & & & & & & & & \\ \hline
    & & & & & & & & \\
    & \Big{0} & & & \Big{0} & & & \Big{0} & \\
    & & & & & & & &
  \end{array}\right].
\end{multline}
}

\end{example}

\section{Symmetries of the bases}
\label{sec:symm}

An $n\times n$ matrix $A = (a_{ij})$ is \emph{centrosymmetric} if
$a_{ij} = a_{n-i+1,n-j+1}$ for all $i,j$. It is \emph{skew-centrosymmetric}
if $a_{ij} = -a_{n-i+1,n-j+1}$ for all $i,j$. 

\begin{lemma}
  Let 
  \begin{align*}
    \cs_n &= \{\boldv\in \lsvec_{n,n}:\, \boldv\text{ is centrosymmetric}\} \text{ and}\\
    \scs_n &= \{\boldv\in \lsvec_{n,n}:\, \boldv\text{ is skew-centrosymmetric}\}.
  \end{align*}
  Then $\lsvec_{n,n} = \cs_n \oplus \scs_n$.
\end{lemma}
\begin{proof}
  The proof relies on the same technique used to show that any space
  of matrices splits into symmetric and skew-symmetric parts. Define a
  ``rotation-by-180-degrees'' map $\theta:\, \lsvec_{n,n} \rightarrow
  \lsvec_{n,n}$ by sending the matrix $A = (a_{ij})\in \lsvec_{n,n}$
  to $\theta(A) = (a_{n-i+1,n-j+1})$. Then the matrix $\csop(A) = (A +
  \theta(A))/2 \in \cs_n$ and $\scsop(A) = (A-\theta(A))/2\in
  \scs_n$. Furthermore, $A = \csop(A) + \scsop(A)$ and $\cs_n \cap
  \scs_n$ is the singleton set consisting of the $n\times n$ zero
  matrix.
\end{proof}

The basis $B_{n,n}$ for $\lsvec_{n,n}$ naturally splits into
centrosymmetric and skew-centrosymmetric pieces. More precisely, the
basis vector $\bu^i\bu^j \in \cs_n$ if and only if either both $\bu^i$
and $\bu^j$ are symmetric or if neither are. We can decompose
$\lsvec_{n,n}$ further by considering the symmetric and skew-symmetric
parts of matrices. For example, we can replace each pair of basis
vectors $\{\bu^i\bu^j,\bu^j\bu^i\}$ for $i\neq j$ with the pair
\begin{equation*}
  \left\{\frac{\bu^i\bu^j+(\bu^i\bu^j)^T}{2},\frac{\bu^i\bu^j-(\bu^i\bu^j)^T}{2}\right\}.
\end{equation*}

\begin{example}
  If the above replacements are performed on $B_{3,3}$, we obtain the basis
  \begin{equation*}
  \left\{
    \begin{bmatrix}
      1 & -2 & 1\\
      -2 & 4 & -2\\
      1 & -2 & 1
    \end{bmatrix}, \quad 
    \begin{bmatrix}
      1 & -1 & 0\\
      -1 & 0 & 1\\
      0 & 1 & -1
    \end{bmatrix},  \quad 
    \begin{bmatrix}
      0 & 1 & -1\\
      -1 & 0 & 1\\
      1 & -1 & 0
    \end{bmatrix},  \quad 
    \begin{bmatrix}
      1 & 0 & -1\\
      0 & 0 & 0\\
      -1 & 0 & 1
    \end{bmatrix}
  \right\}.
  \end{equation*}
  Note that the first and last are symmetric matrices lying in
  $\cs_n$; the second and third both lie in $\scs_n$, but the second
  is symmetric while the third is skew-symmetric. The subspace of
  $\lsvec_{3,3}$ consisting of matrices that are both centrosymmetric
  and skew-symmetric is zero-dimensional. Also, note that the second
  and third basis vectors are, in fact, zeroed Latin squares.
  Figure~\ref{fig:symmpic} illustrates the analogous basis for $n=7$.
\end{example}

\begin{figure}
  \includegraphics[scale=.6]{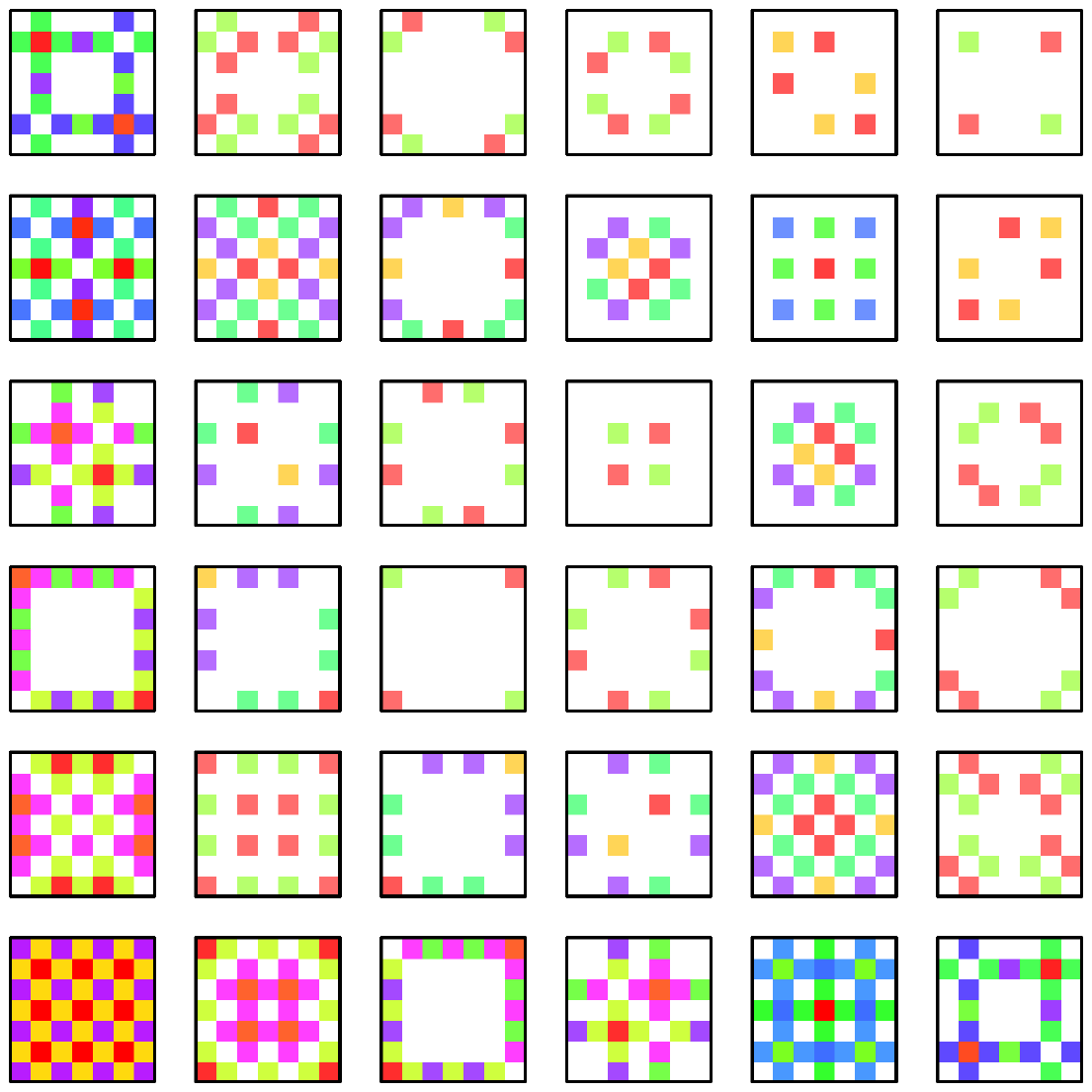}
  \caption{Basis for $\lsvec_{7,7}$ that has been decomposed in
    symmetric and skew-symmetric parts. Numbers have been replaced
    with colors in order to highlight the various symmetries.}\label{fig:symmpic}
\end{figure}

\section{Future directions}
\label{sec:future}

As mentioned in the introduction, there is a simple, non-orthogonal
basis for $\lsvec_{m,n}$. For $1\leq a\leq m-1$ and $1\leq b \leq
n-1$, let $F_{a,b} = (f_{ij})$ be the $m\times n$ matrix of all zeros
except for $f_{a,b}=f_{a+1,b+1}=1$ and $f_{a+1,b}=f_{a,b+1}=-1$. It is
trivial to see that each $F_{a,b}$ lies in $\lsvec_{m,n}$. Also, as
each of the $(m-1)(n-1)$ matrices $F_{a,b}$ has a distinct ``northwest
corner'', it follows that they are linearly independent and hence a
basis. They are not, in general, orthogonal. However, they do satisfy
the important properties required by a \emph{Markov basis}. Roughly:
Fix marginals $\br$ and $\bc$ and define a graph $G(\br,\bc)$ whose
vertices are all contingency tables in $T(\br,\bc)$. Add an edge
between two vertices differing by $\pm F_{a,b}$. The resulting graph
can be shown to be connected. It turns out that one can construct a
random ($2$-way) contingency table with given marginals by taking a
random walk on $G(\br,\bc)$.

If we try to construct an analogous Markov chain using the elements of
$B_{m,n}$ as our basis, we immediately run into a problem: Not every
contingency table is a $\mathbb{Z}$-linear combination of the elements
of $B_{m,n}$. Since the number of contingency tables with fixed
marginals is finite, we can address this by suitably rescaling the
elements of $B_{m,n}$. However, if we do this, then even if a given
$\mathbb{Z}$-linear combination lies in $T(\br,\bc)$, it might not
correspond to an actual contingency table (i.e., after performing the
shift to the affine plane, it might not have integer entries).

Nonetheless, this process might be worth further exploration. A bound
on the necessary amount of scaling can be found by expressing $B_{m,n}$ in
terms of the basis $\{F_{a,b}\}$. Using the correspondingly scaled
elements of $B_{m,n}$, a random walk could be taken and then integer
programming used to find the closest contingency table. However, it is
unclear whether the orthogonality of the $B_{m,n}$ is worth these
complications.

\begin{question}
  Is there any benefit to constructing a Markov chain based on the
  elements of the orthogonal basis $B_{m,n}$ rather than on the
  (non-orthogonal) basis $\{F_{a,b}\}$.
\end{question}

\subsection{The polytope of Latin squares}

Viewing any zeroed Latin square as a vector in $\mathbb{R}^{n^2}$, its
norm squared is the square pyramidal number $n\sum_{i=1}^n i^2 =
n^2(n^2-1)/12$. Let $\overline{B}_{n,n}$ be the orthonormal basis for
$\lsvec_{n,n}$ obtained by normalizing the elements of $B_{n,n}$. It
follows that the coordinates $(c_{ij})$ of any zeroed Latin square
with respect to $\overline{B}_{n,n}$ lie on a sphere in $\lsvec_{n,n}$
centered at the origin whose radius squared is $n^2(n^2-1)/12$.

\begin{question}
  Is there a nice characterization of the convex polytope whose
  vertices correspond to the order-$n$ Latin squares?
\end{question}

Two references pertaining to Latin squares arising in the context of
familiar polytopes are~\cite{Betten,Fisk}. In
Figure~\ref{fig:latinpoly} we illustrate the dual of the polytope for
$n=3$ as visualized by Polymake~\cite{polymake} and
POV-Ray~\cite{povray} (the dual was chosen as we found it to be less
visually confusing). Analogous questions could be asked for other
combinatorial sets such as normal magic squares or Sudoku boards.

\begin{figure}
  \includegraphics[scale=.15]{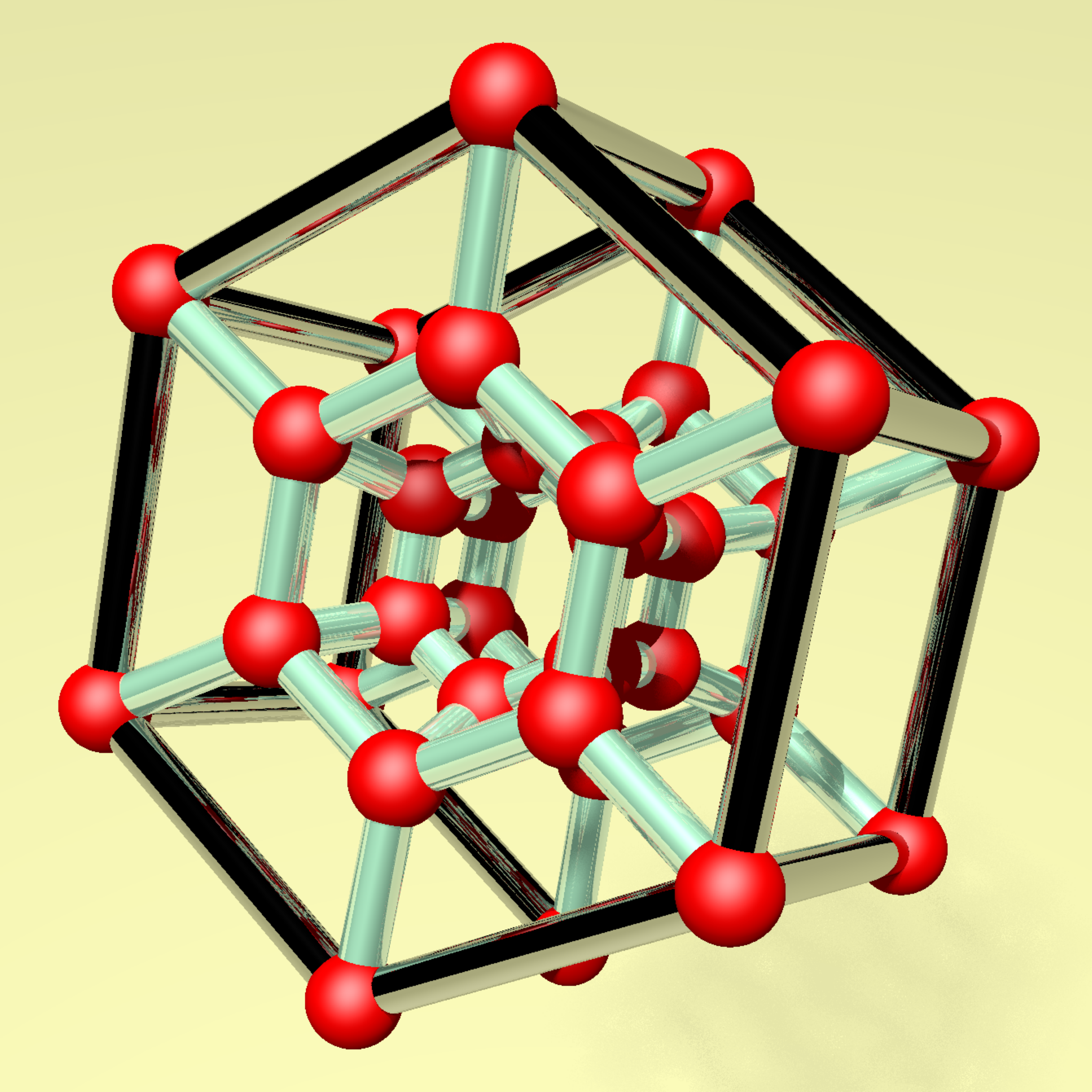}
  \caption{Visualization of dual of the polytope formed by the convex
    hull of the twelve order-$3$ Latin squares.}\label{fig:latinpoly}
\end{figure}

\subsection{Properties of the coordinates}

Given a Latin square, there are numerous transformations of it that
will lead to new Latin squares. For instance, we might rotate or
reflect the square around an appropriate axis or permute rows or
columns. Or, if there is a $2\times 2$ subarray of the form
$\begin{bmatrix}a & b\\b & a\end{bmatrix}$ (an \emph{intercalate}),
  then we can exchange the positions of these $a$'s and $b$'s to get a
  new Latin square. Or we can consider various \emph{conjugates} of
  the Latin square by permuting the triples $(i,j,x_{ij})$. It would
  be interesting to understand these operations in terms of the
  coordinates instead.

  Alternatively, we could investigate operations on coordinates that
  lead to new Latin squares. For example, notice that for the
  order-$3$ Latin squares, if we consider the coordinate vectors up to
  sign, there are only two possibilities: $(0,\pm 1/2,\pm 1/2,0)$ and
  $(\pm 1/4,\pm 1/4,\pm 1/4,\pm 3/4)$. For the 161,280 order-$5$ Latin
  squares, there are only 4,665 possibilities; each equivalence
  class has at least $16$ elements. It is not clear how these
  equivalences are related, in general, to the transformations
  considered in the previous paragraph. Hopefully further
  investigation will shed light on issues such as the observation that
  the number of Latin squares is divisible by a surprisingly high
  power of $2$ (see ~\cite{alter,mullen,mckay-wanless}).

\begin{remark}
  A Sage worksheet containing code to construct the orthogonal bases
  described in this paper can be found at the author's web
  page~\cite{sagecode}.
\end{remark}

\section*{Acknowledgments}

The author gratefully acknowledges helpful discussions with Sara
Billey, Jeff Buzas, Jeff Dinitz, John Schmitt, Dan Velleman and Jill
Warrington.

\bibliography{square-bases}{}
\bibliographystyle{hplain}

\end{document}